\documentclass[11pt,english,a4paper]{smfart}

\usepackage[french,main=english]{babel}
\usepackage[utf8, latin1]{inputenc}

\usepackage{graphicx}
\usepackage{ae,amsfonts,euscript,enumerate}
\usepackage{amssymb}
\usepackage{amsmath}
\usepackage[cm]{aeguill}
\usepackage{xcolor}

\renewcommand{\theequation}{\thesection.\arabic{equation}}

\newcommand{\Z}{\mathbb{Z}}
\newcommand{\Q}{\mathbb{Q}}

\newcommand{\K}{\mathbb{K}}

\newcommand{\Qbar}{\overline{\mathbb{Q}}}
\newcommand{\C}{\mathbb{C}}

\newcommand{\Exc}{\mathrm{Exc}}

\newcommand{\kS}{\mathfrak{S}}
\newcommand{\cL}{\mathcal{L}}
\newcommand{\Ba}{\mathbf{a}}
\newcommand{\Bb}{\mathbf{b}}
\newcommand{\Hyp}{F_{\Ba,\Bb}}
\newcommand{\Si}{\mathrm{Si}}

\newtheorem{thm}{Theorem}[section]

\newtheorem{defn}[thm]{Definition}

\newtheorem{prop}[thm]{Proposition}

\newtheorem{lem}[thm]{Lemma}
\newtheorem{cor}[thm]{Corollary}

\newtheorem{thmx}{Theorem}

\numberwithin{equation}{section}
\usepackage{fullpage}
\usepackage[T1]{fontenc}
\usepackage{graphicx}
\usepackage{amsmath}
\usepackage{amsfonts}
\usepackage{amssymb}
\usepackage{a4wide}

\author{\'E. Delaygue}
\address{
Univ Lyon, Universit\'e Claude Bernard Lyon 1, CNRS UMR 5208, Institut Camille Jordan, F-69622 Villeurbanne Cedex, France}
\email{delaygue@math.univ-lyon1.fr}

\title{A Lindemann-Weierstrass theorem for $E$-functions}
\date{}

\thanks{This project has received funding from the ANR project De Rerum Natura (ANR-19-CE40-0018).} 



\begin{document}

\begin{abstract}
$E$-functions were introduced by Siegel in 1929 to generalize Diophantine properties of the exponential function. After developments of Siegel's methods by Shidlovskii, Nesterenko and Andr\'e, Beukers proved in 2006 an optimal result on the algebraic independence of the values of $E$-functions which generalizes the Lindemann-Weierstrass theorem. Since then, it seems that no general result was stated concerning the relations between the values of a single $E$-function. We prove that Andr\'e's theory of $E$-operators and Beuker's result lead to a Lindemann--Weierstrass theorem for $E$-functions in its ``linear independence'' formulation. As a consequence, we show that all transcendental values at algebraic arguments of an entire hypergeometric function are linearly independent over $\Qbar$.
\end{abstract}

\maketitle

\section{Introduction}

Hermite proved in 1873 that $e$ is a transcendental number. Nine years later, Lindemann generalized this result and obtained a statement that contains the transcendence of both numbers $e$ and $\pi$. 

\begin{thmx}[Hermite--Lindemann]\label{thm: HL}
If $\alpha$ is a non-zero algebraic number then $e^\alpha$ is transcendental.
\end{thmx}

 In his original memoir of 1882, Lindemann sketched a more general result which was rigorously proved by Weierstrass and published in 1885.

\begin{thmx}[Lindemann--Weierstrass]\label{thm: LW}
  If $\alpha_1,\dots,\alpha_n$ are algebraic numbers that are linearly independent over $\mathbb{Q}$, then $e^{\alpha_1},\dots,e^{\alpha_n}$ are algebraically independent over $\Q$.
\end{thmx}

This theorem has the following equivalent ``linear independence'' formulation \cite{Baker}, which is the main motivation of this article.

\begin{thmx}[Lindemann--Weierstrass]\label{thm: LWB}
  If $\alpha_1,\dots,\alpha_n$ are distinct algebraic numbers, then $e^{\alpha_1},\dots,e^{\alpha_n}$ are linearly independent over $\Qbar$.
\end{thmx}

In order to generalize the Lindemann--Weierstrass theorem to other special functions, Siegel introduced in 1929 \cite{Siegel} the notion of \textit{$E$-function}. He proved in particular a result analogous to Theorem \ref{thm: HL} for the Bessel function $J_0$: if $\alpha$ is a non-zero algebraic number then 
$$
J_0(\alpha)=\sum_{n=0}^\infty\frac{(-1)^n}{n!n!}\left(\frac{\alpha}{2}\right)^{2n} 
$$
is transcendental. Siegel also obtained an analogue of Theorem \ref{thm: LW}:
\begin{quote}
  \textit{If $\alpha_1^2,\dots,\alpha_n^2$ are pairwise distinct non-zero algebraic numbers, then $J_0(\alpha_1),\dots,J_0(\alpha_n)$ are algebraically independent over $\Q$}.
\end{quote}
The additional condition $\alpha_i\neq\pm\alpha_j$ is natural for the Bessel function since $J_0(-z)=J_0(z)$.

\begin{defn}\label{defn: E}
  An \textit{$E$-function} is an entire function $f(z)$ that satisfies a non-trivial linear differential equation with coefficients in $\Qbar[z]$ and which has a power series expansion of the form
  $$
  f(z)=\sum_{n=0}^\infty\frac{a_n}{n!}z^n
  $$
  whose coefficients $a_n$ are algebraic numbers and satisfy the following growth condition: letting $d_n\geq 1$ denote, for every $n\geq 0$, the smallest integer such that $d_na_0,\dots,d_na_n$ are algebraic integers, there exists a real number $C>0$ such that, for all $\sigma\in\mathrm{Gal}(\Qbar/\Q)$ and all $n\geq 1$, we have $|\sigma(a_n)|\leq C^n$ and $d_n\leq C^n$.
\end{defn}

Observe that the growth conditions on the $a_n$'s in Definition \ref{defn: E} are equivalent to $h(a_0,\dots,a_n)=O(n)$ for all $n$, where $h$ denotes the absolute logarithmic height. The definition of $E$-functions has been first given by Siegel in 1929 in a more general way. Classical examples of $E$-functions are the exponential, Bessel functions and more generally hypergeometric functions (\textit{e.g.} \cite{ShidlovskiiB}) 
\begin{equation}\label{eq: Hyp}
  \Hyp(z^{s-r}):=\sum_{n=0}^\infty\frac{(a_1)_n\cdots(a_r)_n}{(b_1)_n\cdots(b_s)_n}z^{(s-r)n},
\end{equation}
with integers $s>r\geq 0$, $\Ba=(a_1,\dots,a_r)\in\mathbb{Q}^r$ and $\Bb=(b_1,\dots,b_s)\in(\Q\setminus\mathbb{Z}_{\leq 0})^s$, with the convention, for $r=0$, that an empty product is equal to $1$. Galochkin gave in \cite{Galochkin81} a necessary and sufficient condition for a generalized hypergeometric function to be an $E$-function, which essentially reduces to the rationality of the parameters $a_i$ and $b_j$. Observe that our notation $\Hyp$ slightly differs from the standard notation for generalized hypergeometric functions. The set of $E$-functions forms a $\Qbar[z]$-algebra stable by transformations $f(z)\mapsto f(\lambda z)$ with $\lambda\in\Qbar$. However, the $\Qbar[z]$-algebra generated by hypergeometric $E$-functions does not contain all $E$-functions (see \cite{FJ21}). It is also useful to have in mind that an $E$-function is either a polynomial or transcendental over $\C(z)$.
\medskip

The methods of Siegel were successively improved by Shidlovskii \cite{Shidlovskii}, Nesterenko--Shidlovskii \cite{NS96}, Andr\'e \cite{AndreI}  and Beukers \cite{Beukers} to finally reach in 2006 the following deep result.

\begin{thmx}[Beukers \cite{Beukers}]\label{thm: Beukers}
  Let $Y=(f_1,\dots,f_n)^\top$ be a vector of $E$-functions satisfying a system of first order equations $Y'=AY$ where $A$ is an $\medmuskip=0mu n\times n$-matrix with entries in $\Qbar(z)$. Denote the common denominator of the entries of $A$ by $T(z)$. Let $\alpha$ be an algebraic number satisfying $\alpha T(\alpha)\neq 0$. Then, for any homogeneous polynomial $P\in\Qbar[X_1,\dots,X_n]$ such that $P(f_1(\alpha),\dots,f_n(\alpha))=0$, there exists a polynomial $Q\in\Qbar[Z,X_1,\dots,X_n]$, homogeneous in the variables $X_1,\dots,X_n$, such that $Q(\alpha,X_1,\dots,X_n)=P(X_1,\dots,X_n)$ and $Q(z,f_1(z),\dots,f_n(z))=0$.
\end{thmx}

Andr\'e \cite{AndreB} generalized Theorem~\ref{thm: Beukers} to $E$-functions in Siegel's sense, see also \cite{Lepetit}. Since then, it seems that no general result was stated concerning the relations between the values of a single $E$-function. As in the case of the exponential function, we can separate the diophantine results on the values of a given $E$-function $f$ into two steps. 

First, we consider in Section \ref{sec: Hermite-Lindemann} an analogue of Theorem \ref{thm: HL}, that is to determine the algebraic numbers $\alpha$ such that $f(\alpha)$ is algebraic. This problem was completely solved by Adamczewski and Rivoal in \cite{AR18}. 

Then, being given algebraic points $\alpha_1,\dots,\alpha_n$ such that $f(\alpha_i)$ is transcendental, $1\leq i\leq n$, we use Beukers' lifting result and Andr\'e's theory of $E$-operators to prove a criterion for the linear independence over $\Qbar$ of the values $1,f(\alpha_1),\dots,f(\alpha_n)$. It generalizes Theorem \ref{thm: LWB} to $E$-functions. 
\medskip

To every $E$-function $f$ corresponds a $G$-series $\psi(f)$ defined formally on the Taylor expansions at $0$ by
$$
\psi\left(\sum_{n=0}^\infty\frac{a_n}{n!}z^n\right)=\sum_{n=0}^\infty a_nz^n.
$$
As pointed out by Andr\'e in \cite{AndreI}, the operator $\psi$ can be obtained \textit{via} the Laplace transform. If $f$ is an $E$-function, then there exists a positive constant $\rho$ such that $|f(z)|=O(e^{\rho|z|})$. Hence the Laplace transform $f^+(z)=\int_0^\infty f(w)e^{-wz}\mathrm{d}w$ is well-defined for $\Re(z)>\rho$. A standard calculation shows that $\psi(f)(z)=\frac{1}{z}f^+(\frac{1}{z})$. The series $\psi(f)$ has a positive radius of convergence and satisfies a non-trivial linear differential equation over $\Qbar[z]$. It follows that $\psi(f)$ has finitely many singularities at finite distance, the set of which we denote by $\kS(f)$. Our main result is the following.

\begin{thm}\label{thm: main}
  Let $f_1,\dots,f_n$ be $E$-functions with pairwise disjoint sets $\mathfrak{S}(f_i)$. Let $\alpha$ be an algebraic number such that $f_i(\alpha)$ is transcendental for all $i$, $1\leq i\leq n$. Then the numbers $1, f_1(\alpha),\dots,f_n(\alpha)$ are linearly independent over $\Qbar$.
\end{thm}

The conclusion of Theorem \ref{thm: main} can be rephrased as follows: any non-trivial $\Qbar$-linear combination of $f_1(\alpha),\dots,f_n(\alpha)$ is a transcendental number. We obtain a generalization of Theorem~\ref{thm: LWB} by setting $f_1=\dots=f_n$ in Theorem \ref{thm: main}.

\begin{cor}\label{cor: LWD}
  Let $f$ be an $E$-function. Let $\alpha_1,\dots,\alpha_n$ be non-zero algebraic numbers such that, for all $i\neq j$, $f(\alpha_i)$ is transcendental and $\alpha_i/\alpha_j\neq\rho_1/\rho_2$ for all $\rho_1,\rho_2\in\kS(f)$. Then the numbers $1,f(\alpha_1),\dots,f(\alpha_n)$ are linearly independent over $\Qbar$.
\end{cor}

In the case of the exponential function, we have $\psi(\exp)=1/(1-z)$ so $\kS(\exp)=\{1\}$. The non-nullity assumption on the $\alpha_i$'s in Corollary \ref{cor: LWD} is counterbalanced by the presence of $1$ in the conclusion. Using Theorem \ref{thm: HL}, we retrieve the linear independence of exponentials evaluated at distinct algebraic numbers. 

As an application, we obtain the following new diophantine result on values of hypergeometric series.

\begin{thm}\label{thm: HyperLW}
Let $F_{\Ba,\Bb}$ be an entire hypergeometric function, with $\Ba\in\Q^{r}$, $\Bb\in(\Q\setminus\Z_{\leq 0})^{s}$, where $r$ and $s$ are integers satisfying $s>r\geq 0$. Let $\alpha_1,\dots,\alpha_n$ be pairwise distinct algebraic numbers such that, for all $i$, $1\leq i\leq n$, $F_{\Ba,\Bb}(\alpha_i)$ is transcendental. Then the numbers $1, F_{\Ba,\Bb}(\alpha_1),\dots,F_{\Ba,\Bb}(\alpha_n)$ are linearly independent over $\Qbar$.
\end{thm}

Theorem \ref{thm: HyperLW} is a corollary of Theorem \ref{thm: Hyp2} stated in Section \ref{sec: applications}. Observe that the exponential function corresponds to the parameters $r=0$, $s=1$ and $\Bb=(1)$.

\section{Diophantine properties of the values of $E$-functions}

\subsection{Hermite--Lindemann and $E$-functions}\label{sec: Hermite-Lindemann}

In contrast to $\exp$ and $J_0$, a transcendental $E$-function $f$ can take algebraic values at non-zero algebraic arguments. We call \textit{exceptional values} those non-zero algebraic numbers $\alpha$ such that $f(\alpha)\in\Qbar$. For example \cite{BRS22}, for every $a\in\Q\setminus\Z_{\leq 0}$ and every positive integer $d$, the transcendental hypergeometric $E$-function
$$
\sum_{n=0}^\infty\binom{n+d}{d}\frac{1}{(a+d+1)_n}z^n
$$
has $d$ exceptional values. The set of exceptional values of an entire function $f$ is denoted by $\Exc(f)$.  We say that $f$ is \textit{purely transcendental} if it has no exceptional value.

Since $f$ satisfies a non-trivial linear differential equation over $\Qbar[z]$, there exists a number field $\mathbb{K}$ which contains all its Taylor coefficients. Based on Beukers' lifting result, one can show \cite{FR16} that for any exceptional value $\alpha$, $f(\alpha)\in\K(\alpha)$. More precisely, Bostan, Rivoal and Salvy proved (see \cite{RnoP} and \cite{BRS22}) the following \textit{canonical decomposition}: every transcendental $E$-function can be written in a unique way as $f=p+qg$ with $p,q\in\Qbar[z]$, $q$ monic and $q(0)\neq 0$, $\deg(p)<\deg(q)$ and $g$ is a purely transcendental $E$-function. In particular, the exceptional values of $f$ are the roots of $q$. A result analogous to Theorem \ref{thm: HL} for $f$ thus comes down to finding its canonical decomposition or at least the roots of $q$.

By applying Beukers' lifting result to the system associated with the minimal inhomogeneous differential equation for $f$, Adamczewski and Rivoal \cite{AR18} proved that there exists an algorithm to perform the following tasks. Given an $E$-function $f$ as input, it first says whether $f$ is transcendental or not. If it is transcendental, it then outputs $\Exc(f)$. The input of the algorithm is a differential equation together with sufficiently many initial conditions to define the $E$-function $f$, but an oracle has to guarantee that $f$ is indeed an $E$-function. This algorithm was recently improved in \cite{BRS22}. Together with Rivoal's notes \cite{RnoP}, this yields an algorithm to compute the canonical decomposition of any given $E$-function.

\subsection{A Lindemann--Weierstrass theorem for $E$-functions}

As explained above, the first obstruction to generalize Theorem \ref{thm: LWB} is that at most one of the algebraic numbers $\alpha_1,\dots,\alpha_n$ can be an exceptional value for $f$. The second obstruction is that an $E$-function may satisfy difference equations such as $J_0(-x)=J_0(x)$ or $\sin(-x)=-\sin(x)$, so we must avoid the relations $\alpha_i\neq\pm\alpha_j$ in those cases. More generally, given any $E$-function $f$ and any non-zero distinct algebraic numbers $\alpha_1,\dots,\alpha_n$, we can construct an $E$-function
$$
g(z)=L_1(z)f\left(\frac{z}{\alpha_1}\right)+\cdots+L_n(z)f\left(\frac{z}{\alpha_n}\right),
$$
where $L_i$ is the Lagrange polynomial of degree $n$ such that $L_i(\alpha_j)$ is the Kronecker delta $\delta_{i j}$. It follows that $g(\alpha_1)=\cdots=g(\alpha_n)=f(1)$. As Theorem \ref{thm: main} shows, this situation is in fact encoded by the singularities of the underlying $G$-functions.
\medskip

Under the assumtions of Theorem \ref{thm: main}, it is not hard to prove that the $E$-functions $f_1,\dots,f_n$ are $\Qbar(z)$-linearly independent unless two of them are polynomials. However, to naively prove the $\Qbar$-linear independence of the values $f_1(\alpha),\dots,f_n(\alpha)$ with Beukers' lifting result, one has to construct an appropriate differential system which possibly brings new functions or has non-zero singularities. Somehow, Theorem \ref{thm: main} shows that those difficulties can essentially be overcome.   

As a corollary, we obtain the following result on the linear independence of the values of $E$-functions evaluated at different algebraic arguments.

\begin{cor}\label{cor: LWDn}
  Let $f_1,\dots,f_n$ be $E$-functions. Let $\alpha_1,\dots,\alpha_n$ be non-zero algebraic numbers such that, for all $i\neq j$, $f_i(\alpha_i)$ is transcendental and  $\alpha_i/\alpha_j\neq \rho_1/\rho_2$ for all $\rho_1\in\kS(f_i)$, $\rho_2\in\kS(f_j)$. Then the numbers $1, f_1(\alpha_1),\dots,f_n(\alpha_n)$ are linearly independent over $\Qbar$.
\end{cor}

\begin{proof}[Proof of Corollary \ref{cor: LWDn}]
  We set $g_i(z):=f_i(\alpha_iz)$ for all $i$. We also have $\psi(g_i)(z)=\psi(f_i)(\alpha_iz)$ so $\kS(g_i)$ is the set of the numbers $\rho/\alpha_i$ for $\rho$ in $\kS(f_i)$. If $\rho_1\in\kS(f_i)$ and $\rho_2\in\kS(f_j)$, then $\rho_1/\alpha_i\neq\rho_2/\alpha_j$ by hypothesis. In particular the sets $\kS(g_1),\dots,\kS(g_n)$ are pairwise disjoint and we can apply Theorem \ref{thm: main} to $g_1,\dots,g_n$ with $\alpha=1$.
\end{proof}

We obtain the generalization of Theorem \ref{thm: LWB}, Corollary \ref{cor: LWD} above, by setting $f_1=\dots=f_n$ in Corollary \ref{cor: LWDn}.
\medskip

Let us illustrate Corollary \ref{cor: LWD} with another example. By Siegel's theorem, the Bessel function $J_0$ is purely transcendental and we have 
$$
\psi(J_0)=\sum_{n=0}^\infty\binom{2n}{n}\left(\frac{-z^2}{4}\right)^n=\sum_{n=0}^\infty\frac{(1/2)_n}{n!}\big(-z^2\big)^n=\frac{1}{\sqrt{1+z^2}},
$$
so $\kS(J_0)=\{-i,i\}$. We retrieve the $\Qbar$-linear independence of $J_0(\alpha_1),\dots,J_0(\alpha_n)$ when $\alpha_i\neq\pm\alpha_j$. We describe in Section \ref{sec: applications} several applications of Theorem \ref{thm: main} and its corollaries, in particular we obtain that the transcendental values at algebraic arguments of an entire hypergeometric $E$-function are linearly independent over $\Qbar$.

We prove in Section \ref{sec: pull} another straightforward consequence of Corollary \ref{cor: LWDn}: given transcendental $E$-functions $f_1,\dots,f_n$, and non-zero algebraic functions $\mu_1,\dots,\mu_n,\eta_1,\dots,\eta_n$ over $\Qbar(z)$ such that the $\eta_i$'s are non-constant and pairwise linearly independent over $\Qbar$, then the function
$$
\mu_1\cdot (f_1\circ\eta_1)+\cdots+\mu_n\cdot (f_n\circ\eta_n)
$$
is transcendental over $\Qbar(z)$. This is a broad generalization of the fact that $E$-functions cannot be solutions of non-trivial linear difference equations over $\Qbar(z)$ associated with a Mahler or a shift operator (\textit{e.g.} \cite{ADH21}).

\section{Proof of Theorem \ref{thm: main}}

Besides the main tools of the proof of Theorem \ref{thm: main} that are Beukers' lifting result and Andr\'e's theory of $E$-operators, we shall use the following two lemmas.

\begin{lem}\label{lem: sing dif}
Let $f$ be an $E$-function and $\cL$ a linear differential operator with coefficients in $\Qbar[z]$. Then the singularities at finite distance of $\psi(\cL f)$ are singularities of $\psi(f)$, that is $\kS(\cL f)\subseteq \kS(f)$.
\end{lem}

\begin{proof}
First observe that if $f_1$ and $f_2$ are two $E$-functions, then $\kS(f_1+f_2)$ is a subset of $\kS(f_1)\cup\kS(f_2)$. So we can reduce the proof to the case where $\mathcal{L}=\beta z^k(d/dz)^\ell$ with $\beta\in\Qbar$ and $k,\ell\in\Z_{\geq 0}$. Hence it suffices to prove the lemma for the differential operators defined by $\beta\in\overline{\mathbb{Q}}$, $z$ and $\frac{d}{dz}$.
Calculations give the useful formulas (\textit{e.g.} \cite[Section 2]{AndreI})
$$
\psi(zf(z))=\left(z^2\frac{d}{dz}+z\right)\psi(f)\quad\textup{and}\quad \psi\left(\frac{d}{dz}f(z)\right)=\frac{\psi(f)(z)-f(0)}{z}.
$$
In all cases, we see that the singularities at finite distance of $\psi(\mathcal{L} f)$ are singularities of $\psi(f)$, so $\kS(\mathcal{L}f)\subseteq\kS(f)$. 
\end{proof}

The proof of Theorem \ref{thm: main} rests on the following well-known result.

\begin{lem}\label{lem: no sing is pol}
A $G$-function with no singularity at finite distance is a polynomial function with coefficients in $\Qbar$.
\end{lem}

In the case of $G$-functions in the strict sense, that are series $\psi(f)$ with $f$ an $E$-function as defined by Definition \ref{defn: E}, Lemma \ref{lem: no sing is pol} can be demonstrated directly using the product formula in number fields. However, there is a more elaborated proof for the general definition given by Siegel in 1929, which we present below for the sake of completeness.

\begin{proof}
Let $f$ be a $G$-function with no singularity at finite distance. Hence $f$ has an infinite radius of convergence and it is an entire function. Let $\cL$ be the minimal linear differential operator annihilating $f$. By Lepetit's result \cite[Th\'{e}or\`{e}me 1]{Lepetit}, which is a generalization of the Andr\'{e}--Chudnovsky--Katz theorem for $G$-functions in Siegel's original sense, every singularity of $\cL$ in $\mathbb{P}^1(\mathbb{C})$ is a regular singularity. In particular, $f$ has at most polynomial growth at infinity by Fuch's criterion. But an entire function with polynomial growth is a polynomial by Liouville's theorem, so $f$ is a polynomial. 
\end{proof}

We can now prove our main result.

\begin{proof}[Proof of Theorem \ref{thm: main}]
Let $f_1,\dots,f_n$ be $E$-functions such that, for all distinct $i$ and $j$, the sets $\kS(f_i)$ and $\kS(f_j)$ are disjoint. Let $\alpha$ be a fixed non-zero algebraic number such that $f_1(\alpha),\dots,f_n(\alpha)$ are transcendental numbers. Let us assume that there is a $\overline{\mathbb{Q}}$-linear relation between the numbers $1,f_1(\alpha),\dots, f_n(\alpha)$ given by
\begin{equation}\label{eq: LinRel}
\lambda_0+\lambda_1f_1(\alpha)+\cdots+\lambda_nf_n(\alpha)=0.
\end{equation}
We shall show that this relation is trivial, that is $\lambda_0=\cdots=\lambda_n=0$.
\medskip

We write $f_0$ for the constant $E$-function equal to $1$. Let $i$ be fixed in $\{0,\dots,n\}$. By the results of Andr\'e on $E$-operators \cite{AndreI}, there exists a differential operator $L_i$ annihilating $f_i$ and with only $0$ and $\infty$ as singularities. We write $s_i$ for the differential order of $L_i$ and $Y_i=(f_i,\dots,f_i^{(s_i-1)})^\top$. It follows that $Y_i$ is solution of a differential system $Y'=A_iY$ where $A_i$ is a matrix with entries in $\Qbar[z,1/z]$. The direct sum of those systems gives a system $Y'=AY$, where the entries of $A$ belong to $\Qbar[z,1/z]$, and with the concatenation of the $Y_i$'s as a vector solution of $E$-functions. We consider the polynomial $P$ in $\Qbar[(X_{i,j})_{0\leq i\leq n,0\leq j< s_i}]$ given by
$$
P((X_{i,j})_{0\leq i\leq n,0\leq j< s_i})=\lambda_0X_{0,0}+\lambda_1X_{1,0}+\cdots+\lambda_nX_{n,0},
$$
so that $P((f_i^{(j)}(\alpha))_{0\leq i\leq n,0\leq j<s_i})=0$.
 Since $\alpha$ is non-zero, it is not a singularity of $A$ and we can apply Theorem \ref{thm: Beukers}: there exists a polynomial $Q$ in $\Qbar[Z,(X_{i,j})_{0\leq i\leq n,0\leq j<s_i}]$, linear in the $X_{i,j}$'s, such that
$$
Q(z,(f_i^{(j)}(z))_{0\leq i\leq n,0\leq j<s_i})=0
$$
and  
$$
Q(\alpha,(X_{i,j})_{0\leq i\leq n,0\leq j< s_i})=\lambda_0X_{0,0}+\lambda_1X_{1,0}+\cdots+\lambda_nX_{n,0}.
$$
This means that there are polynomials $R_{i,j}(z)\in\Qbar[z]$, $0\leq i\leq n$, $0\leq j< s_i$, such that
$$
\sum_{i=0}^n\sum_{j=0}^{s_i-1}R_{i,j}f_i^{(j)}=0,
$$
and which satisfy $R_{i,0}(\alpha)=\lambda_i$ and $R_{i,j}(\alpha)=0$ for all $i$ in $\{0,\dots n\}$ and all $j\geq 1$. For every $i$ in $\{0,\dots,n\}$, we consider a linear differential operator $\mathcal{L}_i$ with coefficients in $\Qbar[z]$ satisfying
$$
\mathcal{L}_if_i=\sum_{j=0}^{s_i-1}R_{i,j}f_i^{(j)}.
$$
It follows that $\mathcal{L}_0f_0+\dots+\mathcal{L}_nf_n= 0$, with $(\mathcal{L}_if_i)(\alpha)=\lambda_if_i(\alpha)$ for all $i$, $0\leq i\leq n$. Since $\psi$ is clearly additive, we also have
$$
\psi(\mathcal{L}_0f_0)+\dots+\psi(\mathcal{L}_nf_n)= 0.
$$

By Lemma \ref{lem: sing dif}, for every $i$ in $\{0,\dots,n\}$, the singularities at finite distance of $\psi(\mathcal{L}_if_i)$ are singularities of $\psi(f_i)$, so $\kS(\mathcal{L}_if_i)\subseteq\kS(f_i)$. It follows that for all distinct $i$ and $j$, the sets $\kS(\mathcal{L}_if_i)$ and $\kS(\mathcal{L}_jf_j)$ are disjoint.

For every $i$ in $\{0,\dots,n\}$, we obtain that 
\begin{equation}\label{eq: division}
\psi(\mathcal{L}_if_i)=-\underset{j\neq i}{\sum_{j=0}^n}\psi(\mathcal{L}_jf_j)
\end{equation}
is a $G$-function with no singularity at finite distance. Indeed, Equation \eqref{eq: division} yields
$$
\kS(\mathcal{L}_if_i)\subseteq\underset{j\neq i}{\bigcup_{j=0}^n}\kS(\mathcal{L}_j f_j).
$$ 
Since the sets $\kS(\mathcal{L}_0 f_0),\dots,\kS(\mathcal{L}_n f_n)$ are pairwise disjoint, we obtain that $\kS(\mathcal{L}_i f_i)$ is empty. 

Now by Lemma \ref{lem: no sing is pol}, a $G$-function with no singularity at finite distance is a polynomial function with coefficients in $\Qbar$. 

Hence, for every $i$ in $\{0,\dots,n\}$, $\psi(\mathcal{L}_if_i)\in\Qbar[z]$ and we also have $\mathcal{L}_if_i\in\Qbar[z]$. As a consequence, we obtain that
$$
\lambda_if_i(\alpha)=(\mathcal{L}_if_i)(\alpha)\in\Qbar.
$$
For all $i$ in $\{1,\dots,n\}$, $f_i(\alpha)$ is transcendental so $\lambda_i=0$. It follows that $\lambda_0=\lambda_1=\cdots=\lambda_n=0$ and the relation is trivial.
\end{proof}

\section{Applications of Theorem \ref{thm: main}}\label{sec: applications}

\subsection{Hypergeometric $E$-functions}

We recall that hypergeometric $E$-functions $\Hyp(z^{s-r})$ are defined by \eqref{eq: Hyp} and that $(a)_n$ denotes the Pochhammer symbol defined by $(a)_0=1$ and $(a)_n=a(a+1)\cdots(a+n-1)$ for positive integers $n$. Since the only singularity at finite distance of a hypergeometric $G$-function is $1$, we obtain the following Diophantine results for their values.
 
\begin{thm}\label{thm: Hyp1}
  For all $i\in\{1,\dots,n\}$, let $s_i>r_i\geq 0$ be integers, set $k_i=s_i-r_i$ and consider parameters $\Ba_i\in\Q^{r_i}$ and $\Bb_i\in(\Q\setminus\Z_{\leq 0})^{s_i}$. Let $\alpha_1,\dots,\alpha_n$ be non-zero algebraic numbers such that, for all $i\neq j$, $F_{\Ba_i,\Bb_i}(\alpha_i)$ is transcendental and $\alpha_i^{k_j}/\alpha_j^{k_i}\neq(k_j/k_i)^{k_i k_j}$. Then the numbers $1,F_{\Ba_1,\Bb_1}(\alpha_1),\dots,F_{\Ba_n,\Bb_n}(\alpha_n)$ are linearly independent over $\Qbar$.
\end{thm}

\begin{proof}
  We consider $f_i(z)=F_{\Ba_i,\Bb_i}(z^{k_i})$ which is a hypergeometric $E$-function. The $G$-series $\psi(f_i)(z)$ is given by
  $$
  \sum_{n=0}^\infty (k_in)!\frac{(a_{i,1})_n\cdots(a_{i,r_i})_n}{(\beta_{i,1})_n\cdots(\beta_{i,s_i})_n}z^{k_in}=\sum_{n=0}^\infty k_i^{k_in} \left(\frac{1}{k_i}\right)_n\cdots\left(\frac{k_i}{k_i}\right)_n\cdot\frac{(a_{i,1})_n\cdots(a_{i,r_i})_n}{(\beta_{i,1})_n\cdots(\beta_{i,s_i})_n}z^{k_in},
  $$
  which is equal to $g_i((k_iz)^{k_i})$ where $g_i$ is a hypergeometric $G$-function with only $1$ as finite singularity. It follows that the singularities at finite distance of $\psi(f_i)$ are of the form $\rho/k_i$ where $\rho$ is a $k_i$-th root of unity. For all $i$, we write $\beta_i$ for a $k_i$-th root of $\alpha_i$. If $\rho_1\in\kS(f_i)$ and $\rho_2\in\kS(f_j)$ are such that $\beta_i/\beta_j=\rho_1/\rho_2$ then there are respective $k_i$-th and $k_j$-th roots of unity $\xi_i$ and $\xi_j$ such that $\beta_i/\beta_j=\xi_ik_j/(\xi_j k_i)$. By rising this equality to the power $k_ik_j$, we obtain that $\alpha_i^{k_j}/\alpha_j^{k_i}=(k_j/k_i)^{k_i k_j}$, which is a contradiction. Since $f_i(\beta_i)=F_{\Ba_i,\Bb_i}(\alpha_i)$ is transcendental, the result follows by Corollary~\ref{cor: LWDn}. 
\end{proof}

By focusing on parameters with same differences $s_i-r_i$, we obtain that the transcendental values of the associated hypergeometric functions are $\Qbar$-linearly independent.

\begin{thm}\label{thm: Hyp2}
  Let $k$ be a positive integer. For all $i\in\{1,\dots,n\}$, let $r_i\geq 0$ and $s_i=r_i+k$ be integers and consider parameters $\Ba_i\in\Q^{r_i}$ and $\Bb_i\in(\Q\setminus\Z_{\leq 0})^{s_i}$. Let $\alpha_1,\dots,\alpha_n$ be pairwise distinct algebraic numbers such that, for all $i$, $1\leq i\leq n$, $F_{\Ba_i,\Bb_i}(\alpha_i)$ is transcendental. Then the numbers $1, F_{\Ba_1,\Bb_1}(\alpha_1),\dots,F_{\Ba_n,\Bb_n}(\alpha_n)$ are linearly independent over $\Qbar$.
\end{thm}

\begin{proof}
The $\alpha_i$'s are non-zero because $F_{\Ba_i,\Bb_i}(\alpha_i)$ is transcendental for all $i$, $1\leq i\leq n$. We set $f_i(z):=F_{\Ba_i,\Bb_i}(z^k)$ for all $i$. The singularities at finite distance of each $\psi(f_i)$ are of the form $\rho/k$ where $\rho$ is a $k$-th root of unity. Let $\beta_1,\dots,\beta_n$ be respective $k$-th roots of $\alpha_1,\dots,\alpha_n$. For all distinct $i$ and $j$, the ratio $\beta_i/\beta_j$ is not a $k$-th root of unity because $\alpha_i\neq\alpha_j$. The result follows by Corollary~\ref{cor: LWDn}.
\end{proof}

The linear and algebraic independence of the values of hypergeometric $E$-functions and their derivatives have been extensively studied. General independence criteria have been obtained by Salikhov \cite{Salikhov90, Salikhov98}, Galochkin \cite{Galochkin11} and recently Gorelov \cite{Gorelov16,Gorelov21}. Those results are stated with some restrictions on the parameters $\Ba_i$ and $\Bb_i$ or under the assumption that the $\alpha_i$'s are linearly independent over $\Q$. As far as we know, the absence of restrictions in our assumptions yields new results.

\subsection{Algebraic pullbacks of $E$-functions} We obtain the following functional result as a straightforward application of Corollary \ref{cor: LWDn}.\label{sec: pull}

\begin{prop}
Let $f_1,\dots,f_n$ be transcendental $E$-functions. Let $\mu_1,\dots,\mu_n,\eta_1,\dots,\eta_n$ be non-zero algebraic functions over $\Qbar(z)$ such that the $\eta_i$'s are non-constant and pairwise linearly independent over $\Qbar$. Then the function
\begin{equation}\label{eq: pull}
\mu_1\cdot (f_1\circ\eta_1)+\cdots+\mu_n\cdot (f_n\circ\eta_n)
\end{equation}
is transcendental over $\Qbar(z)$.
\end{prop}

\begin{proof}
A non-zero algebraic function $\eta$ has finitely many zeros. If in addition $\eta$ is non-constant, then for each algebraic number $\alpha$, the equation $\eta(z)=\alpha$ has only finitely many solutions. If $\eta_i$ is well defined at $\alpha\in\Qbar$, then $\eta_i(\alpha)\in\Qbar$. Furthermore, $f_i$ has only finitely many exceptional values. Thus there is a non-empty open subset $\mathcal{O}$ of $\C$ such that the $\mu_i$'s and $\eta_i$'s are well defined and do not vanish on $\mathcal{O}$, and for every algebraic number $\alpha\in\mathcal{O}$ and every $i$, $f_i(\eta_i(\alpha))$ is a transcendental number. Assume that the function \eqref{eq: pull} is algebraic over $\Qbar(z)$. Hence, for every algebraic numbers $\alpha\in\mathcal{O}$, we have
$$
\mu_1(\alpha)f_1(\eta_1(\alpha))+\cdots+\mu_n(\alpha)f_n(\eta_n(\alpha))\in\Qbar.
$$
By Corollary \ref{cor: LWDn}, for each such $\alpha$, there are $i\neq j$, $\rho_1\in\kS(f_i)$ and $\rho_2\in\kS(f_j)$ such that $\eta_i(\alpha)/\eta_j(\alpha)=\rho_1/\rho_2$. Since the $f_i$'s are transcendental, the sets $\kS(f_i)$ are finite and non-empty. We obtain by the pigeonhole principle that there exist $i\neq j$, $\lambda\in\Qbar$ and infinitely many $\alpha$ in $\mathcal{O}$ such that $\eta_i(\alpha)/\eta_j(\alpha)=\lambda$. This yields $\eta_i=\lambda\eta_j$ and ends the proof.
\end{proof}

\subsection{A last illustration of Theorem \ref{thm: main}}

Let us consider a last application of Corollary~\ref{cor: LWDn} to the sine integral.

\begin{prop}\label{prop: Si}
  Let $\alpha_1^2,\dots,\alpha_n^2,\beta_1^2,\dots,\beta_n^2$ be pairwise distinct algebraic numbers. Then the numbers
  $$
  \int_{\alpha_i}^{\beta_i}\frac{\sin t}{t}dt,\quad 1\leq i\leq n,
  $$
  are transcendental and linearly independent over $\Qbar$. 
\end{prop}

\begin{proof}
Let us consider the sine integral and its associated $G$-series
$$
\Si(z)=\int_0^z\frac{\sin t}{t}dt\quad\textup{and}\quad \psi(\Si)(z)=\sum_{n=0}^\infty\frac{(-1)^n}{2n+1}z^{2n+1},
$$
which is the Taylor expansion of $\arctan$ at the origin. It follows that $\kS(\Si)=\{-i,i\}$. Let $j$ be fixed in $\{1,\dots,n\}$. By \cite[Chapter 8]{ShidlovskiiB}, the sine integral 
is a purely transcendental $E$-function. Hence, if one of the $\alpha_j$'s or $\beta_j$'s is zero, the corresponding integral is transcendental. Otherwise, we apply Corollary \ref{cor: LWD} to obtain that the numbers $1,\Si(\alpha_j)$ and $\Si(\beta_j)$ are $\Qbar$-linearly independent. In particular, $\Si(\beta_j)-\Si(\alpha_j)$ is a transcendental number, as expected. Now it suffices to apply Corollary \ref{cor: LWD} to obtain that the numbers $\Si(\alpha_1),\dots,\Si(\alpha_n),\Si(\beta_1),\dots,\Si(\beta_n)$ are $\Qbar$-linearly independent.
\end{proof}

\noindent\textbf{Acknowledgements}. We very warmly thank Boris Adamczewski, Tanguy Rivoal and Julien Roques for numerous discussions on various aspects of this project.

\end{document}